\documentclass[11pt,english]{smfart} 
\usepackage[T1]{fontenc}
\usepackage[english,french]{babel}
\usepackage{amscd}
\usepackage{amssymb,url,xspace,smfthm}
\usepackage{amsbsy,bm} 
\usepackage{paralist}
\usepackage{datetime} 
\usepackage{colortbl}
\usepackage{color}
\usepackage[dvipsnames]{xcolor}
\usepackage{mathrsfs}  
\usepackage[colorlinks,	colorlinks,
linkcolor=BrickRed,
citecolor=Green,
urlcolor=Cerulean,hypertexnames=true]{hyperref} 
\usepackage{mathtools}
\usepackage{euscript}
\usepackage[all]{xy}
\usepackage{hyperref}
\usepackage{graphicx} 
\usepackage[text={150mm,251mm},centering, marginparwidth=75pt]{geometry} 
\usepackage{comment}  
\usepackage{cite}  
\usepackage{thmtools}
\usepackage{enumitem}
\usepackage{letltxmacro} 
\usepackage{nameref}
\usepackage{cleveref}
\usepackage{calc}
\usepackage{interval}
\usepackage{manfnt}
\usepackage{tikz-cd}  
\usepackage[utf8]{inputenc}
\usepackage{marginnote}
\usepackage{smfhyperref}
\usepackage[hyperpageref]{backref}

\usepackage{faktor} 
\usepackage{scalerel}

\makeatletter     
\newcommand{\crefnames}[3]{%
	\@for\next:=#1\do{%
		\expandafter\crefname\expandafter{\next}{#2}{#3}%
	}%
}
\makeatother

\setlist[itemize]{wide = 0pt, labelwidth = 2em, labelsep*=0em, itemindent = 0pt, leftmargin = \dimexpr\labelwidth + \labelsep\relax, noitemsep,topsep = 1ex,}
\setlist[enumerate]{wide = 0pt, labelwidth = 2em, labelsep*=0em, itemindent = 0pt, leftmargin = \dimexpr\labelwidth + \labelsep\relax, noitemsep,topsep = 1ex}

\theoremstyle{plain}
\newtheorem{thmx}{Theorem}
\renewcommand{\thethmx}{\Alph{thmx}} 
\newtheorem{theorem}{Theorem}[section]  
\newtheorem{lemma}[theorem]{Lemma}
 
\newtheorem{proposition}[theorem]{Proposition}
\newtheorem{corollary}[theorem]{Corollary}

\theoremstyle{definition}
 
\theoremstyle{remark}
\newtheorem{remark}[theorem]{Remark} 
\newtheorem{example}[theorem]{Example}

\numberwithin{equation}{section}  

\theoremstyle{plain}
\newlist{thmlist}{enumerate}{1}
\setlist[thmlist]{wide = 0pt, labelwidth = 2em, labelsep*=0em, itemindent = 0pt, leftmargin = \dimexpr\labelwidth + \labelsep\relax, noitemsep,topsep = 1ex, font=\normalfont, label=(\roman*), ref=\thetheorem.(\roman{thmlisti})}

\addtotheorempostheadhook[theorem]{\crefalias{thmlisti}{thm}}

\addtotheorempostheadhook[assumption]{\crefalias{thmlisti}{assumption}}

\addtotheorempostheadhook[corollary]{\crefalias{thmlisti}{cor}}

\addtotheorempostheadhook[proposition]{\crefalias{thmlisti}{proposition}}

\addtotheorempostheadhook[definition]{\crefalias{thmlisti}{dfn}}

\addtotheorempostheadhook[lemma]{\crefalias{thmlisti}{lem}}
\addtotheorempostheadhook[main]{\crefalias{thmlisti}{main}}

\addtotheorempostheadhook[remark]{\crefalias{thmlisti}{rem}}

\newlist{thmenum}{enumerate}{1} 
\setlist[thmenum]{wide = 0pt, labelwidth = 2em, labelsep*=0em, itemindent = 0pt, leftmargin = \dimexpr\labelwidth + \labelsep\relax, noitemsep,topsep = 1ex, font=\normalfont, label=(\roman*), ref=\thethmx.(\roman{thmenumi})}
\crefalias{thmenumi}{thmx} 

\newlist{corlist}{enumerate}{1} 
\setlist[corlist]{wide = 0pt, labelwidth = 2em, labelsep*=0em, itemindent = 0pt, leftmargin = \dimexpr\labelwidth + \labelsep\relax, noitemsep,topsep = 1ex, font=\normalfont, label=(\roman*), ref=\thecorx.(\roman{corlisti})}
\crefalias{corlisti}{corx} 

\crefname{lemma}{Lemma}{Lemmas} 
\crefname{conjecture}{Conjecture}{Conjectures}
\crefname{theorem}{Theorem}{Theorems}
\crefname{proposition}{Proposition}{Propositions}
\crefname{definition}{Definition}{Definitions}
\crefname{remark}{Remark}{Remarks}
\crefname{corollary}{Corollary}{Corollaries}
\crefname{corx}{Corollary}{Corollaries}
\crefname{problem}{Problem}{Problems}
\crefname{thmx}{Theorem}{Theorems}
\crefname{claim}{Claim}{Claims}
\crefname{assumption}{Assumption}{Assumptions}
\crefname{main}{Main Theorem}{Main Theorems}

\newcommand{\C}{\mathbb{C}}
\newcommand{\R}{\mathbb{R}}
\newcommand{\Q}{\mathbb{Q}}
\newcommand{\nilp}{\mathrm{nilp}}
\newcommand{\Gal}{\mathrm{Gal}}
\newcommand{\Alb}{\mathrm{Alb}}
\newcommand{\Z}{\mathbb{Z}}
\newcommand{\fkg}{\mathfrak{g}}
\newcommand{\calu}{\mathcal{U}}
\newcommand{\Tor}{\mathrm{Tor}}
\newcommand{\im}{\mathrm{im}}
\newcommand{\ab}{\mathrm{ab}}
\newcommand{\Hom}{\mathrm{Hom}}
\newcommand{\Mal}{\mathrm{Mal}}
\newcommand{\orb}{\mathrm{orb}}

\newcommand{\pbar}{\overline{\partial}}

\begin{document}

\title[The holomorphic convexity of nilpotent coverings]{On the holomorphic convexity of nilpotent coverings over compact K\"ahler surfaces}

{
\author{Yuan Liu}
\email{lexliu@hku.hk}
\address{Department of Mathematics, The University of Hong Kong, Pokfulam Road, Hong
Kong.} 
\urladdr{https://sites.google.com/view/yuan-lius-website}
}

\date{\today}
\makeatletter
\def\subjclassname{\textup{2020} Mathematics Subject Classification}
\makeatother

\begin{abstract}
We prove that every regular covering over a compact
K\"ahler surface with a virtually nilpotent Galois
group is holomorphically convex, provided that it
does not have two ends.
Our result is optimal with respect to the dimension of the base manifold, as counterexamples exist in complex dimensions at least three.

Furthermore, we prove that the Malcev covering of any compact K\"ahler manifold has at most one end, yielding, in particular, the free indecomposability of the nontrivial Malcev quotient of any K\"ahler group.
\end{abstract}

\maketitle

\section{Introduction}

The Shafarevich conjecture asks whether the universal covering over a compact K\"ahler manifold is holomorphically convex. A general program toward resolving this conjecture was initiated independently by Koll\'ar \cite{Kollar,Kollar_book} and Campana \cite{Campana}, and has since been extensively developed by Claudon, Eyssidieux, Katzarkov, Napier, Pantev, Ramachandran, and many others; see, for example, \cite{Napier90,Katzarkov97,KatzarkovRamachandran98,Eys,EKPR,CCE}. We also refer the reader to the survey by Eyssidieux \cite{Eys11}. For analogous developments in the quasi-projective setting, see, for example, \cite{AACampana,DYK,DY2024,BBT,GreenGriffithsKatzarkov,Deng26,Bakker26ICM}.

Beyond universal coverings, it is natural to ask whether intermediate coverings of compact K\"ahler manifolds are holomorphically convex, and to identify conditions under which this property holds. Several results have been established in this direction, including the following:
\begin{itemize}[label=\textbullet]
    \item the holomorphic convexity of Malcev coverings over smooth projective manifolds, proved by Katzarkov \cite{Katzarkov97};
    \item the holomorphic convexity of reductive coverings over smooth projective surfaces without two ends, proved by Katzarkov and Ramachandran \cite{KatzarkovRamachandran98};
    \item the holomorphic convexity of Malcev coverings over compact K\"ahler manifolds, established by Leroy \cite{LeroyThesis} (see also \cite{Claudon});
    \item the holomorphic convexity of Galois coverings over smooth projective manifolds associated with absolutely constructible subsets of the Betti moduli space, established by Eyssidieux \cite{Eys};
    \item the holomorphic convexity of reductive coverings over compact K\"ahler surfaces without two ends, established by the author \cite{Yuan23}.
\end{itemize}
The present paper continues the author's previous work by investigating the holomorphic convexity of Galois coverings of compact K\"ahler surfaces whose deck transformation groups are virtually nilpotent.

Recall that a regular covering of a compact K\"ahler manifold is called a $\mathcal P$-covering if its Galois group satisfies the property $\mathcal P$, such as being infinite, nilpotent, or linear. Throughout this paper, all coverings are assumed to be regular. We say that a group is virtually $\mathcal P$ if it contains a normal subgroup of finite index satisfying $\mathcal P$. Since every finite covering of a compact K\"ahler manifold is compact, and hence holomorphically convex, it suffices to consider infinite coverings. Notice that for holomorphic convexity, we only need to consider the infinite coverings, and a virtually $\mathcal{P}$-covering is holomorphically convex if and only if the corresponding $\mathcal{P}$-covering is. In particular, the study of holomorphic convexity
for virtually $\mathcal P$-coverings reduces to that
for $\mathcal P$-coverings, after passing to a finite
\'etale covering of the base.

Our restriction to complex surfaces is motivated not only by technical considerations (see \cref{rmk: why two dim} and \cref{thm: Napier's criterion}), but also by the existence of counterexamples in complex dimensions at least three; see \cref{ex: optimal on dimension two}. The counterexample is obtained by embedding of Cousin's classical example (\cite{Cousin1910}, see \cref{ex: Cousin}) into higher-dimensional manifold. Some discussion of two-ended coverings is given in \cref{subsec: two ends}, along with a re-explanation of a theorem of Napier--Ramachandran \cite{NR01GAFA} using BNS invariant in \cref{thm: two ends BNS}.

Our first main result is the following.
\begin{thmx}[\cref{thm: nilpotent covering}]\label{thmx: A}
Every virtually nilpotent covering of a compact
K\"ahler surface that does not have two ends
is holomorphically convex.
\end{thmx}

Nilpotent groups represent the next level of
algebraic complexity beyond abelian groups.
Accordingly, our proof of the above theorem
is motivated by the argument in the abelian case
(see \cref{prop: abelian cover}).
The earlier approaches to Malcev coverings
in \cite{Katzarkov97,Leroy} make essential use
of mixed Hodge structures or higher Albanese maps.
In contrast, our argument works directly with
a group-theoretic nilpotent quotient of the
fundamental group. To this end, we employ
analytic methods based on flat connections,
following Brudnyi \cite{Brudnyi99,Brudnyi};
see the proof of \cref{lem: finiteness propagation for tf nilp}.

We first recall two known results that establish holomorphic convexity from the existence of a weakly plurisubharmonic exhaustion function under
additional assumptions in \cref{subsec: Lemmas for preparation}. The first applies when the degeneracy locus is ``small'', namely, when the exhaustion function is generically strictly plurisubharmonic; this is Narasimhan's theorem, see \cref{lem: theorem of narasimhan}. The second addresses the case where the degeneracy locus is ``large'', say, we have sufficiently many regular level sets of the exhaustion function which are Levi-flat; this is a theorem of Napier--Ramachandran, see \cref{lem: Levi-flat case NR}.

Now we briefly outline the proof of \cref{thmx: A}. Let $X$ be a compact K\"ahler surface and let
\[
\rho:\Gamma=\pi_1(X)\twoheadrightarrow D
\]
be a surjective homomorphism onto an infinite
torsion-free nilpotent group $D$.
Consider the induced homomorphism
\[
\rho_{\ab}:
\Gamma\xrightarrow{\rho}D
\xrightarrow{q}A,
\]
where $A=D_{\ab}/\Tor(D_{\ab})$ is the maximal
free-abelian quotient of $D$.
By the abelian case (\cref{prop: abelian cover}),
the corresponding covering $X_{\ab}$ is
holomorphically convex.
We may therefore consider its Cartan--Remmert
reduction \cite{Cartan60}
(see \cref{eq: CR reduction})
\[
r:X_{\ab}\to R.
\]

We divide the proof into two cases according
to the dimension of $R$.

When $\dim_{\C}R=2$, Napier's criterion \cite{Napier90} (see \cref{thm: Napier's criterion}), together with
a nilpotent analogue of the Lasell--Ramachandran theorem \cite{LasellRamachandran} (\cref{cor: finiteness propagation for nilp}), allows us to establish the holomorphic convexity of $X_\rho$. This case is treated in \cref{subsec: Napier's criterion on holomorphic convexity}.

The case $\dim_{\C}R=1$ requires a substantially
different argument. By taking the quotient of
$r:X_{\ab}\to R$ by the induced $A$-actions, we obtain a holomorphic fibration $f:X\to C$. Equipping $C$ with the orbifold structure induced
by $f$, we obtain the well-known exact sequence
\[
1\to N\to
\Gamma\xrightarrow{f_*^{\orb}}
\pi_1^{\orb}(C)\to 1,
\]
where $N$ is finitely generated. On the other hand, we have the natural exact sequence
\[
1\to\ker q\to
D\xrightarrow{q}A\to 1.
\]
The homomorphism $\rho_{\ab}$ factors through
$\pi_1^{\orb}(C)$. Consequently,
\[
G:=\rho(N)\trianglelefteq D
\]
is contained in $\ker q$, and hence $D/G$
surjects onto the infinite group $A$.
In particular, $G$ is a normal subgroup of
infinite index in $D$. The structure of this subgroup and its interaction with the orbifold fibration play a central role in the proof.

Rather than working directly on $X_\ab$ or $X_\rho$,
we introduce an intermediate covering between them. Roughly speaking, we choose the covering so that (the quotient of) $G$ there is central in $D$. The construction combines information from the abelian quotient with the lift $f_1:X_1\to C_1$ of the holomorphic fibration $f:X\to C$ to an infinite covering corresponding to the quotient group $D/G$.

To construct a plurisubharmonic exhaustion function
on this intermediate covering, we combine functions
arising from the abelian data with the pullback of
a strictly subharmonic exhaustion function on the
corresponding noncompact Riemann surface $C_1$, which is Stein. By analyzing the degeneracy locus, we show that the resulting exhaustion function is generically strictly plurisubharmonic. Narasimhan's theorem then implies the holomorphic convexity of the intermediate covering, from which the desired conclusion for $X_\rho$ follows from Napier's criterion. We refer to \cref{subsec: Cartan-Remmert reduction is one-dimensional} for the details.

In addition, we investigate the number of ends
of Malcev coverings of compact K\"ahler manifolds
in \cref{sec: On the number of ends of Malcev coverings}.
We establish the following result.

\begin{thmx}[\cref{thm: one end malcev}]
The Malcev covering of a compact K\"ahler manifold
has at most one end. More precisely,
\[
e(X^{\Mal})=
\begin{cases}
0, & \text{\quad if }b_1(X)=0,\\
1, & \text{\quad if }b_1(X)>0.
\end{cases}
\]
\end{thmx}

The Malcev covering may be viewed as the maximal
regular covering detected by torsion-free
nilpotent quotients of the fundamental group.
Our theorem can thus be regarded as a nilpotent
analogue of the classical one-endedness result
for universal coverings of compact K\"ahler
manifolds \cite{ABR,Gromov}.

As a group-theoretic consequence
(\cref{cor: malcev kahler obstruction}),
we obtain that the Malcev quotient
$\Gamma^\Mal$ of a K\"ahler group $\Gamma$ admits no nontrivial free-product decomposition. Finally, the one-endedness statement for Malcev coverings does not extend to arbitrary compact complex manifolds, as illustrated by \cref{ex: Hopf}.

\section{Holomorphic convexity of abelian coverings}

In this section, we prove that every abelian covering
of a compact K\"ahler surface is holomorphically convex, provided that it does not have two ends. We also establish several lemmas that will be used again in the nilpotent case.

\subsection{Lemmas for preparation}\label{subsec: Lemmas for preparation}

Although the existence of a weakly plurisubharmonic
exhaustion function does not, in general, imply
holomorphic convexity, we present two results in this
section that guarantee holomorphic convexity under
additional assumptions.

The first is the following theorem of Narasimhan,
which is particularly useful in complex dimension two
(see \cite[Lemma 3.1]{KatzarkovRamachandran98}
for a proof).

\begin{lemma}[Narasimhan]\label{lem: theorem of narasimhan}
Let $M$ be a complex manifold and let $g: M \to \R$ be a continuous plurisubharmonic exhaustion function, which is generically strongly plurisubharmonic, i.e., it is strongly plurisubharmonic away from a proper complex analytic subset $Z \subsetneq M$. If all noncompact irreducible components of $Z$ are Stein spaces, then $M$ is holomorphically convex.  
\end{lemma}

\begin{remark}\label{rmk: why two dim}
The first reason to restrict to complex dimension two is that we need to analyze the degeneracy locus of a Levi form in order to apply this theorem. The condition is automatic for $M$ of dimension two, since any noncompact Riemann surface is Stein by the Behnke–Stein theorem.
\end{remark}

The following theorem of Napier--Ramachandran
\cite[Theorem 4.6]{NR95} deals with the case
where the generic regular level set of the exhaustion function is Levi-flat.

\begin{lemma}[Napier--Ramachandran]\label{lem: Levi-flat case NR}
Let $M$ be a connected infinite regular covering of a compact K\"ahler manifold. Assume that $M$ does not have two ends and admits a smooth plurisubharmonic exhaustion function $f$. Suppose that, for every sufficiently large regular value $a$ of $f$, the corresponding nonempty level set $\{x\in M: f(x)=a\}$ is Levi-flat, i.e., the Levi form of $f$ vanishes on its complex tangent space at each point. Then $M$ admits a proper holomorphic mapping onto a Riemann surface; in particular, $M$ is holomorphically convex.
\end{lemma}

A generalization of this result to continuous subanalytic plurisubharmonic exhaustion functions was obtained by the author; see \cite[Theorem in the Introduction]{Yuan23}.

\subsection{Holomorphic convexity on abelian coverings}

The simplest case concerns abelian coverings of compact K\"ahler surfaces that do not have two ends. This result is well known to experts (see
\cite[p.~526]{KatzarkovRamachandran98}\footnote{The statement there is formulated for arbitrary compact K\"ahler manifolds, but Prof.~Ramachandran informed me that this is a typo.}). For completeness, we include a proof here.

\begin{proposition}\label{prop: abelian cover}
Any abelian covering of a compact K\"ahler surface
which does not have two ends is holomorphically convex.
\end{proposition}

\begin{proof}
Let $X$ be a compact K\"ahler surface and
$\Gamma=\pi_1(X)$. Let $\rho:\Gamma\twoheadrightarrow A$ be a surjective
homomorphism onto an abelian group $A$, and let
$\pi: X_\rho=\widetilde X/\ker\rho\to X$ be the corresponding regular abelian covering.

Since $\Gamma$ is finitely generated, $A$ is a
finitely generated abelian group. Hence
$\operatorname{Tor}(A)$ is a finite subgroup. Passing to the covering corresponding to
$\Gamma\twoheadrightarrow A/\operatorname{Tor}(A)$,
and using the fact that holomorphic convexity is
preserved under finite unramified coverings, we may
assume that $A$ is torsion-free. If $A$ is finite,
then $X_\rho$ is compact and the conclusion is
immediate. Since $X_\rho$ does not have two ends, we may assume that
$$\rho:\Gamma\twoheadrightarrow A\cong\mathbb Z^r,
\qquad r\geq 2.$$

For each $1\leq i\leq r$, let
$\chi_i:\Gamma\to\mathbb Z$ be the composition of
$\rho$ with the $i$-th coordinate. Let $\alpha_i$ be the real harmonic representative of the cohomology class corresponding to $\chi_i$. Since the periods of $\alpha_i$ vanish on $\ker\rho$, its pullback to $X_\rho$ is exact. Thus there exists a smooth real-valued function $f_i$ on $X_\rho$ such that $df_i=\pi^*\alpha_i$. Moreover, for each $\gamma\in\Gamma$, we have $f_i(\gamma x)=f_i(x)+\chi_i(\gamma)$. By Hodge theory of compact K\"ahler manifolds, $\alpha_i$ is the real part of a holomorphic one-form. Hence $f_i$ is pluriharmonic and real analytic.

Consider the function
$$g=\sum_{i=1}^r f_i^2:X_\rho\to\mathbb R.$$
We first show that $g$ is an exhaustion. Since
$X_\rho/A=X$ is compact, there exists a compact
subset $K\subset X_\rho$ such that $A\cdot K=X_\rho$.
The map
\[
f=(f_1,\ldots,f_r):X_\rho\to\mathbb R^r
\]
satisfies $f(a x)=f(x)+a$ for every
$a\in A\cong \mathbb Z^r$. Therefore, for any compact
subset $B\subset\mathbb R^r$, only finitely many
translates $aK$ can intersect $f^{-1}(B)$.
Consequently, $f$ is proper, and so is $g$.

Since each $f_i$ is pluriharmonic, we have
\[
\sqrt{-1}\partial\bar\partial g
=2\sqrt{-1}\sum_{i=1}^r
\partial f_i\wedge\bar\partial f_i\geq 0.
\]
Thus $g$ is a real analytic plurisubharmonic
exhaustion function.

We distinguish two cases.
\begin{itemize}
    \item \textbf{Case 1.} Suppose that for some $1\leq j<k\leq r$, $\partial f_j\wedge\partial f_k\not\equiv 0$. Denote
    $$Z=\{x\in X_\rho:
(\partial f_j\wedge\partial f_k)(x)=0\},$$
which is a proper complex analytic subset of $X_\rho$.
Outside $Z$, the forms $\partial f_j$ and
$\partial f_k$ are linearly independent, which implies that $\sqrt{-1}\partial\bar\partial g$ is (strictly) positive definite outside $Z$. Therefore, by \cref{lem: theorem of narasimhan}, $X_\rho$ is holomorphically convex.
\item \textbf{Case 2.}
Suppose that $\partial f_j\wedge\partial f_k\equiv 0$
for all $1\leq j,k\leq r$.

We claim that every nonempty regular level set of
$g$ is Levi-flat. Indeed, let $a$ be a regular value of $g$ such that $M_a=g^{-1}(a)$ is nonempty. At every $x\in M_a$, we have
$$\partial g(x)=2\sum_{i=1}^r
f_i(x)\partial f_i(x)\neq 0.$$
By our assumption, the forms
$\partial f_1(x),\ldots,\partial f_r(x)$
span a complex vector space of dimension at most
one. Hence
\[
\operatorname{Span}_{\mathbb C}
\{\partial f_1(x),\ldots,\partial f_r(x)\}
=\mathbb C\cdot\partial g(x).
\]
Consequently, for every
$v\in T_x^{1,0}M_a=\ker\partial g(x)$,
we have $\partial f_i(v)=0$, for $1\leq i\leq r$. It follows that
\[
\sqrt{-1}\partial\bar\partial g(v,\bar v)
=2\sum_{i=1}^r|\partial f_i(v)|^2=0.
\]
Thus $M_a$ is Levi-flat. Since $r\geq 2$, the covering $X_\rho$ has exactly one end. Moreover, $g$ is a real analytic plurisubharmonic exhaustion, and all its nonempty regular level sets are Levi-flat. Hence \cref{lem: Levi-flat case NR} applies and shows
that $X_\rho$ is holomorphically convex.
\end{itemize}
The proof is now complete.
\end{proof}

\subsection{Explanation on two ends}
\label{subsec: two ends}

We assume that the regular covering does not have two ends in order to exclude the following examples of Cousin \cite{Cousin1910} (see also \cite[Example 3.9]{NR95}).

\begin{example}[Cousin]\label{ex: Cousin}
Let $\Gamma$ be the subgroup of $\C^2$ generated by $(1, 0), (0, 1)$, and $(a, b)$ where $a, b \in \C \backslash\R$. Then the quotient $M = \C^2/\Gamma$ is diffeomorphic to the product of a 3-dimensional real torus and the real line and hence has exactly two ends. It is known that $M$ has a nonconstant holomorphic function if and only if $1, a$, and $b$ are linearly dependent over $\Z$. 
\end{example}

This example of Cousin shows that there exists a
$\Z$-covering of a compact K\"ahler manifold that
admits no nonconstant holomorphic functions and is
therefore not holomorphically convex. This explains
why it is important to exclude such coverings when
studying the holomorphic convexity of intermediate
coverings. In complex dimension two, it turns out
that excluding coverings with exactly two ends is
sufficient for the classes of coverings considered
in this paper. For a regular covering over a compact
manifold, having exactly two ends is equivalent to
its Galois group being virtually $\Z$.

Note that the assumption of not having exactly two
ends is not necessary for holomorphic convexity,
as shown by the following result, which is essentially
a reformulation of \cite[Theorem 4.3]{NR01GAFA}
in terms of Delzant's description \cite{Delzant10}
of the BNS-invariant of K\"ahler groups. For the
definition and properties of the BNS-invariant,
we refer to \cite{BNS}. For a nonzero real character
$\rho:\Gamma\to\R$, we denote by $[\rho]$ its
equivalence class under multiplication by positive
real numbers. The set of all such equivalence
classes is denoted by $S(\Gamma)$, i.e.,
\[
S(\Gamma)=
\bigl(\Hom(\Gamma,\R)\setminus\{0\}\bigr)/\R_{>0}.
\]
For a subgroup $N\subseteq\Gamma$, we denote by
$S(\Gamma,N)$ the subset of $S(\Gamma)$ consisting
of classes $[\rho]$ such that $\rho|_N=0$.

\begin{theorem}\label{thm: two ends BNS}
Let $X$ be a compact K\"ahler manifold with
fundamental group $\Gamma$, and let
$\rho:\Gamma\to\Z$ be a surjective homomorphism.
If $[\rho]\notin\Sigma(\Gamma)$, where
$\Sigma(\Gamma)$ is the BNS-invariant of $\Gamma$,
then the associated covering $X_\rho$ admits a
proper holomorphic map onto a Riemann surface
with connected fibers and is therefore
holomorphically convex.
\end{theorem}
\begin{proof}
By the Theorem B1 in \cite{BNS}, we know that $N=\mathrm{Ker}(\rho)$ is finitely generated if and only if $S(\Gamma,N)=\{[\rho],[-\rho]\}\subseteq \Sigma(\Gamma)$. By the description of the BNS-invariant for K\"ahler groups (Theorem 1.1, \cite{Delzant10}), we know that $\Sigma(\Gamma)$ is symmetric in the sense that $\Sigma(\Gamma)=-\Sigma(\Gamma)$. Thus, we have $N$ is finitely generated if and only if $[\rho]\in\Sigma(\Gamma)$. Now, if $[\rho]\notin\Sigma(\Gamma)$, we have $N$ is not finitely generated. Then by the proof of Theorem 4.3 in \cite{NR01GAFA}, the conclusion is true.
\end{proof}

This phenomenon does not occur in the case of universal coverings. Indeed, any finitely generated group $G$ with exactly two ends is virtually $\Z$, whereas a K\"ahler group cannot be virtually $\Z$ \cite{ABCKT}. However, $\Z$ can occur as a quotient of a K\"ahler group, and hence intermediate coverings with exactly two ends may arise.

Note that Cousin's example can be extended to higher
dimensions by taking products. Thus, in every complex
dimension at least three, there exist abelian (and hence nilpotent) coverings of compact K\"ahler manifolds that are not holomorphically convex, despite having exactly one end. In this sense, the dimensional restriction in our results is optimal.

\begin{example}\label{ex: optimal on dimension two}
Let $M=X_\rho\to X$ be a $\Z$-covering as in
\cref{ex: Cousin}, and let $\C^*\to E$ be a
$\Z$-covering of an elliptic curve. Consider the
product covering
\[X_\rho\times\C^*\to X\times E.\]
Since its Galois group is isomorphic to $\Z^2$,
this covering has exactly one end. On the other hand,
every holomorphic function on $X_\rho\times\C^*$
is constant along the $X_\rho$-factor and hence
factors through the projection onto $\C^*$.
Consequently, the holomorphically convex hull of
any point $(x,z)\in X_\rho\times\C^*$ contains
the noncompact fiber $X_\rho\times\{z\}$.
Therefore, $X_\rho\times\C^*$ is not holomorphically
convex.
\end{example}
This shows that \cref{prop: abelian cover} (and in particular \cref{thm: nilpotent covering}) fails in complex dimensions
greater than or equal to three.

\section{The holomorphic convexity of nilpotent coverings}
In this section, we prove our main theorem
\begin{theorem}\label{thm: nilpotent covering}
Let $X$ be a compact K\"ahler surface with fundamental group $\Gamma$ and let $\rho:\Gamma\twoheadrightarrow D$ be a group homomorphism onto a virtually nilpotent group $D$. Let ${X}_{\rho}$ be the regular covering associated with $\rho$ and assume that it does not have two ends, then ${X}_{\rho}$ is holomorphically convex.
\end{theorem}

Recall that a group $D$ is virtually nilpotent if it contains a (normal) subgroup of finite index which is nilpotent. Since passing to a finite covering does not affect the holomorphic convexity, we can deal with this slightly general situation. Moreover, the torsion subgroup of a finitely generated nilpotent group is finite, thus we can assume without loss of generality that $D$ is \textit{infinite, torsion-free, and nilpotent}. 

To prove this theorem, we need to use the \cref{prop: abelian cover} and consider the abelian covering corresponding to maximal free abelian quotient of $D$; say 
$$\rho_\ab:\Gamma\to D\to A$$
where $A\cong \Z^s$. We claim that $s\ge 2$.

We argue by contradiction and assume $s=1$. By Malcev's theorem (\cite{DrutuKapovich}), $D$ can be embedded into a simply connected nilpotent real Lie group $G$, called the \textit{real Malcev completion} of $D$. Let $\fkg$ be the Lie algebra of $G$. By the universal property of Malcev completion \cite[Appendix A, Corollary 3.8]{Quillen69},
we have
\begin{equation}\label{eq: quillen}
\fkg/[\fkg,\fkg]
\cong D_{\mathrm{ab}}\otimes_{\mathbb Z}\mathbb R.
\end{equation}
and $\dim_\R\fkg/[\fkg,\fkg]= \mathrm{rank}_\Z(D_\ab)=1$. Then $[\fkg,\fkg]=\fkg_2=\fkg_3=\fkg_4=\cdots$, where $\fkg_{i+1}=[\fkg,\fkg_i]$. Since $\fkg$ is nilpotent, we have $\fkg_n=0$ for some $n$, hence $\fkg_2=0$. Now $\fkg\cong \R$ and $D/\Tor\cong \Z$, which contradicts that $D$ does not have two ends. Hence we must have $s\coloneqq \mathrm{rank}_\Z(A)\ge 2$. Then $X_\ab$ does not have two ends and is thus holomorphically convex by \cref{prop: abelian cover}.

Recall that for any complex manifold $Y$ which is holomorphically convex, we can define its \textit{Cartan--Remmert reduction} as follows: for any compact subset $K\subset Y$, by definition, its holomorphically convex hull
$$\hat{K}\coloneqq \{y\in Y: |f(y)|\le \max_{z\in K}|f(z)|\ ,\forall f\in \mathcal{O}(Y) \}$$
is also compact; in particular, the equivalence relation defined by 
$$y_1\sim y_2\Longleftrightarrow f(y_1)=f(y_2),\forall f\in \mathcal{O}(Y)$$
is compact, then the quotient space $R=Y/{\sim}$ admits a natural structure of a Stein complex space
\begin{equation}\label{eq: CR reduction}
r: Y\to R
\end{equation}
and $R$ is called the Cartan--Remmert reduction of $Y$. In particular, $r$ is proper holomorphic with connected fibers and $r^*\mathcal{O}(R)\cong \mathcal{O}(Y)$.

The proof of \cref{thm: nilpotent covering} is divided into two cases, according to whether the Cartan--Remmert reduction of $X_\ab$ has complex dimension two or one. These cases are treated in
\cref{subsec: Napier's criterion on holomorphic convexity} and \cref{subsec: Cartan-Remmert reduction is one-dimensional}, respectively.

\subsection{Napier's criterion on holomorphic convexity}\label{subsec: Napier's criterion on holomorphic convexity}

When the Cartan--Remmert reduction of $X_\ab$ has
complex dimension two, we use the following criterion
of Napier \cite{Napier90} for the holomorphic convexity of complex surfaces. This provides another technical reason for restricting our attention to complex dimension two.

\begin{theorem}[Napier, \textup\textbf{Theorem 6.2,\cite{Napier90}}]\label{thm: Napier's criterion}
Let $Y$ be a noncompact smooth holomorphically convex surface, and suppose that $\pi: \hat Y\to Y$ is a covering surface of $Y$. Assume that
\begin{enumerate}[label=(\roman*)]
    \item The Cartan--Remmert reduction of $Y$ has dimension two, and 
    \item $\hat Y$ contains no noncompact connected analytic subsets whose irreducible components are all compact.
\end{enumerate}
Then $\hat Y$ is holomorphically convex.
\end{theorem}
The analytic subsets in (ii) will be called \textit{an infinite chain}. The existence of such an infinite chain on $\hat Y$ will be an obstruction to its holomorphic convexity. In other words, given (i), the condition (ii) is not only sufficient but also necessary for $\hat Y$ being holomorphically convex.

Now let $X$ be a compact K\"ahler manifold and let
$Z\subset X$ be a connected compact reduced analytic
curve. Suppose that the image of fundamental group of each normalized irreducible component of $Z$ in $\pi_1(X)$ is finite, whereas the image of $\pi_1(Z)\to\pi_1(X)$ is infinite. Then every connected component of the inverse image of $Z$ in the universal covering $\widetilde X$ is a noncompact connected analytic curve whose irreducible components are compact. In particular, it is an infinite chain.

This observation motivates a representation-theoretic version of the same question. Let $X$ be a compact K\"ahler manifold and let $\rho:\pi_1(X)\to G$ be a group homomorphism. We say that a connected compact reduced analytic subvariety $Z\subset X$ has \textit{$\rho$-finiteness propagation} property if the following holds: if the image under $\rho$ of the fundamental group of the \textit{normalization} of each irreducible component of $Z$ is finite, the image
$$\rho\bigl(\operatorname{im}(\pi_1(Z)\to\pi_1(X))\bigr)$$
is also finite.

Lasell and Ramachandran \cite{LasellRamachandran}
shows that the $\rho$-finiteness propagation property holds if $\rho$ is semisimple (along with other results for fields of positive characteristic); for an analogous result in the quasi-projective setting, see the work of Deng and the author in \cite{DengLiu23}. Here, we establish the $\rho$-finiteness propagation property for $\rho$ with $G$ nilpotent. We first show that $\rho$ is torsion-free nilpotent leads to seemingly stronger conclusion that the $\rho$-image of the fundamental group of the compact curve $Z$ is trivial.

The following proof using flat connections is inspired by the work of Brudnyi \cite{Brudnyi}, see also \cite[Part II]{GreenGriffithsKatzarkov}; the reduction from the general curve to a nodal one was used by Katzarkov in \cite{Katzarkov97}.
\begin{lemma}\label{lem: finiteness propagation for tf nilp}
Let $\rho:\pi_1(X)\to G$ be a representation into a torsion-free nilpotent group $D$. Let $Z\subset X$ be a reduced analytic subset such that each of its irreducible components has trivial image under $\rho$, then the image of $\pi_1(Z)$ under $\rho$ is also trivial.
\end{lemma}
\begin{proof}
A combination of theorems by Malcev and Ado--Engel (see \cite[P. 476]{DrutuKapovich}) shows that each finitely generated torsion-free nilpotent group embeds in a real unipotent group $\mathcal{U}_N(\R)$, i.e., the group of upper triangular $N\times N$ matrices with 1 on the diagonal and entries in $\R$. We abuse the notation and continue to denote the composite by $\rho$:
$$\pi_1(X,x)\to D\hookrightarrow \calu_N(\R)\subset \calu_N(\C).$$

Now consider the associated complex vector bundle
$$E\coloneqq \widetilde X\times_\rho \C^N$$
where $\widetilde X$ is the universal covering, and denote $\nabla=d+A$ is a flat connection. Then the monodromy representation with respect to $\nabla$ is $\rho$. By construction, the bundle $E$ has trivial rank one graded local systems. 

We claim that there exists a smooth upper trianglular unipotent gauge transformation under which $A$ satisfies $\partial A=0$. For $A$ an upper triangular unipotent matrix of one-forms, we define the \textit{super-diagonal index} for each entry $A_{ij},\ (j>i)$ as the positive integer $j-i$. Decompose $A=\sum_{k\ge 1} A_k$ according to this index. The claim is proved using induction on $k$. From flatness of connection, we have
$$dA=d\sum_k A_k=-A\wedge A=-\sum_k A_k\wedge \sum_k A_k$$
which implies
\begin{equation}\label{eq: flatness into super-diagonal index components}
dA_k=-\sum_{\substack{p+q=k\\p\ge 1,q\ge 1}}A_p\wedge A_q.
\end{equation}
for each $1\le k\le N-1$. For $k=1$, the sum is empty and $dA_1=0$ is trivially true. Write $A_1=h_1+d\phi_1$ by Hodge decomposition, with $h_1$ harmonic. The gauge transformation $\exp(-\phi_1)$ replaces the $A_1$ by $h_1$, which is \(\partial\)-closed.

Now assume that $\partial A_p=0$ for all $p<k$, then $dA_k$ is $d$-exact, and is $\partial$-closed because
$$\partial dA_k=-\partial \sum_{p+q=k}A_p\wedge A_q=0.$$
Hence $dA_k$ is also $\pbar$-closed; by the $\partial\pbar$-lemma (see \cite{DGMS}), we have $u_k$ such that $dA_k=\partial\pbar u_k$. Thus $d(A_k+\partial u_k)=0$. By the Hodge decomposition of $\Delta_d$, we have
$$A_k+\partial u_k=h_k+d\phi_k$$
for $h_k$ harmonic. Then the gauge transformation $\exp(-\phi_k)$ takes $A_k$ into $h_k-\partial u_k$, which is $\partial$-closed.

Now we prove the lemma. Firstly assume that $Z$ is a nodal curve, and let $Z_i$ be its reduced irreducible components. For $A$ with $\partial A=0$, we claim that $A|_{Z_i^\nu}\equiv 0$, where $Z_i^\nu$ is the normalization of $Z_i$.

By construction, $dA_1=0$, and for any loop $\gamma$ representing an element in $\pi_1(Z_i^\nu)$, we have $\int_\gamma A_1=0$; hence $A_1=du_1$ for some function $u_1$ defined on $Z_i^\nu$. But we have assumed $\partial A_1=0$, then $\partial \pbar u_1=0$ which means $u_1$ is (pluri-)harmonic. By maximum principal, $u_1$ is constant on the smooth compact curve $Z_i^\nu$; and $A_1|_{Z_i^\nu}\equiv 0$ for any $i$. We now proceed by induction on the super-diagonal index to prove $A|_{Z_i^\nu}\equiv 0$. Assume that this is true for $p<k$, then from \cref{eq: flatness into super-diagonal index components}, we have $dA_k\equiv0$ on each $Z_i^\nu$; this along with the fact that $\int_\gamma A_k=0$ for any loop $\gamma$ on $Z_i^\nu$ implies that $A_k$ is $d$-closed, say $A_k=du_k$. By the same argument as for $k=1$, we get $u_k$ is constant on $Z_i^\nu$ and $A_k|_{Z_i^\nu}\equiv 0$. By induction, we have $A|_{Z_i^\nu}\equiv 0$ for each $i$. 

Now consider any element in $\pi_1(Z)$ which is from the dual graph $\Delta_Z$ of $Z$ (see \cite[Lemma 2.1]{Katzarkov97}); in other words, consider 
$$\pi_1(Z)\cong \left(*_{i=1}^n\pi_1(Z_i^\nu)\right)*F_{b_1(\Delta_Z)}$$
and $[\gamma]\in F_{b_1(\Delta_Z)}$ lies the free group part. It is represented by a loop $\gamma$ going through nodal points successively, with portion connecting two adjacent nodal points lying in one of $Z_i$ (equivalently, with interior part in $Z_i^\nu$). By our previous argument, the parallel transport with respect to the connection $\nabla=d+A$ along each such portion is the identity, thus the parallel transport along $\gamma$ is also the identity matrix; in other words, $\rho([\gamma])$ is trivial. Since $\rho$ restricted to $N$ is trivial, along with $\rho|_{\pi_1(Z_i^\nu)}$ is trivial for each $i$, we know $\rho$ is trivial on $\im(\pi_1(Z)\to \pi_1(X))$.

Finally, we prove for the general curve $Z$ by reduction to the nodal case. We can take its normalization $Z'$, say $\pi:Z'\to Z$ which is finite. Denote the preimage of a singular point $p\in Z$ by $\{p_1,\cdots,p_n\}$. By attaching a rational curve $\mathbb P^1$ to $Z^\nu$ at the points $p_1,\ldots,p_m$, using distinct attachment points on $\mathbb P^1$, we obtain a nodal curve $Z''$. There is a natural holomorphic map $Z''\to Z$ which agrees with $\nu$ on $Z^\nu$ and contracts each added rational component to the corresponding singular point of $Z$. Moreover, this map induces a surjection on fundamental groups (see \cite[Lemma 2.4]{Katzarkov97}). The $\rho$-image on each irreducible component of $Z''$ is trivial; hence $\rho(\pi_1(Z'')\to \pi_1(X))$ is trivial, and so is $\rho(\pi_1(Z)\to \pi_1(X))$.
\end{proof}

\begin{corollary}\label{cor: finiteness propagation for nilp}
Let $\rho:\pi_1(X)\to D$ be a representation into a nilpotent group $D$. Let $Z\subset X$ be a connecetd reduced compact analytic curves such that each of its irreducible components has finite image under $\rho$, then the image of $\pi_1(Z)$ under $\rho$ is also finite; i.e., any $Z$ has $\rho$-finiteness propagation.
\end{corollary}
\begin{proof}
Let $\rho_0:\pi_1(X)\to D\to D_0$, where $D_0$ is the torsion-free quotient of $D$, whose kernel is finite \cite[Theorem 13.65]{DrutuKapovich}. By \cref{lem: finiteness propagation for tf nilp}, we know the $\rho_0$-image of each of its irreducible components is trivial, hence $\rho_0(\pi_1(Z))$ is trivial, i.e., $\rho(\pi_1(Z))$ is finite.
\end{proof}

This finiteness propagation property for nilpotent representations allows us to exclude infinite chains of compact curves in the corresponding nilpotent coverings. Combined with Napier's criterion (\cref{thm: Napier's criterion}), it provides the case when $X_\ab$ admits a Cartan--Remmert reduction of dimension two.

\begin{proposition}\label{prop: CR reduction of dim 2}
Let $X$ be a compact K\"ahler surface and $\rho: \pi_1(X)\to D$ be a representation onto a torsion-free nilpotent group $D$. Let $\hat Y=X_\rho\xrightarrow{p} Y\xrightarrow{q} X$ be a tower of infinite regular coverings of $X$, with $Y$ being holomorphically convex. Assume that the Cartan--Remmert reduction of $Y$ is of complex dimension two. Then $\hat{Y}$ is holomorphically convex.
\end{proposition}
\begin{proof}
By \cref{thm: Napier's criterion}, to show $X_\rho$ is holomorphically convex, we only need to show that there is no infinite chain contained in $X_\rho$.

Assume not, let $Z\subset X_\rho$ be an infinite chain. Denote the Cartan--Remmert reduction of $Y$ as $r:Y\to U$. Then $r\circ p(Z)$ is a single point, say $u\in U$. Then $r^{-1}(u)$ is a finite union of compact reduced irreducible curves, say $C=C_1\cup \cdots\cup C_n$. Let $C'$ be the (finite) union of compact reduced irreducible curves appearing in $p(Z)$, i.e., $p(Z)\subset C'$. Each component of $C'$ has a finite image under $\rho$ since it has a compact lifting to $\hat{ Y}=X_\rho$. By \cref{cor: finiteness propagation for nilp}, we have $\rho(\im(\pi_1(C')\to \pi_1(X)))$ is also finite; which contradicts that $Z\subseteq p^{-1}(C')$ is noncompact. Thus there is no infinite chain on $\hat{Y}$, therefore $\hat{Y}$ is holomorphically convex.
\end{proof}

With the notation as defined in \cref{thm: nilpotent covering}, from the above results, we know that if the Cartan--Remmert reduction of $X_\ab$ is of complex dimension two, then by considering the tower of coverings:
$$X_\rho\to X_\ab\to X,$$
we have $X_\rho$ is holomorphically convex.

\subsection{Cartan--Remmert reduction is one-dimensional}\label{subsec: Cartan-Remmert reduction is one-dimensional}

Now we are ready to prove the holomorphic convexity, when the Cartan--Remmert reduction of $X_\ab$ is of dimension one. We still denote the following composition as
$$\rho_\ab: \pi_1(X)\to D\to A\cong \Z^s,$$
and consider the abelian covering $X_\ab$ corresponding to $\rho_\ab$, we have a tower of infinite coverings over $X$:
$$X_\rho\to X_\ab\to X.$$
By \cref{prop: abelian cover}, $X_\ab$ is holomorphically convex. If the Cartan--Remmert reduction of $X_\ab$ has complex dimension two, then by \cref{prop: CR reduction of dim 2}, we have $X_\rho$ is holomorphically convex.

Now we are left to deal with the case when the Cartan--Remmert reduction of $X_\ab$ is one-dimensional, say $r: X_\ab\to C_0$ with $C_0$ a curve. By the construction of the
Cartan--Remmert reduction, the Galois group $A$
acts holomorphically on $C_0$, and $r$ is
$A$-equivariant. Indeed, each element of $A$
acts on $X_\ab$ by a biholomorphic automorphism,
which induces an automorphism of the algebra
$\mathcal O(X_\ab)$ by pullback. Since the fibers
of $r$ are precisely the equivalence classes
of points that cannot be separated by holomorphic
functions, this action preserves the fibers of
$r$ and hence induces a biholomorphic automorphism
of $C_0$. Recall that $r$ is proper, for any compact subset $Z\subset C_0$, $r^{-1}(Z)$ is also compact; by proper discontinuity of $A$ on $X_\ab$, we have $g.(r^{-1}(Z))\cap r^{-1}(Z)\ne \varnothing$ for only finitely many $g\in A$, hence $A$ also acts properly discontinuously on $C_0$. In particular, the action of $A$ on $C_0$ has finite stabilizer and the kernel of $A\to \mathrm{Aut}(C_0)$ is finite; moreover, since $A$ is torsion-free, this kernel must be trivial, i.e., $A$ acts faithfully on $C_0$. Similarly, the finite stabilizer is also trivial. Then $C=C_0/A$ is a compact Riemann surface and $r$ descends to a holomorphic map $f: X\to C$ with connected fibers; this shows that $C$ is compact as the image of $f$. Note that $g(C)\ge 1$.

We equip $C$ with the orbifold structure induced by the multiple fibers of $f$. More precisely, let
$\{p_1,\ldots,p_\ell\}\subset C$ be the points
such that the fiber
\[
f^*(p_i)=\sum_j m_{ij}F_{ij}
\]
has multiplicity $m_i\coloneqq\gcd_j\{m_{ij}\}\geq 2$. We denote the resulting orbifold curve by
\[
C^\orb=
\bigl(C;(p_1,m_1),\ldots,(p_\ell,m_\ell)\bigr).
\]
The fibration $f$ induces a short exact sequence
of fundamental groups:
\begin{equation}\label{eq: ses of hol fibration}
1\to N\to \Gamma\xrightarrow{f_*^\orb} Q\to 1
\end{equation}
where $z\in C$ is a regular value of $f$, $Q=\pi_1^\orb(C)$ and $N=\im(\pi_1(f^{-1}(z)) \to  \pi_1(X))$. In particular, $N$ is finitely generated.

From the associated Lyndon--Hochschild--Serre spectral sequence, we obtain:
\begin{equation}\label{eq: LHS les}
\cdots\to  H^1(\Gamma;\R)\to H^1(N;\R)^Q\to H^2(Q;\R)\to H^2(\Gamma;\R)\to \cdots
\end{equation}
It is well known that the last map is injective,
as a consequence of the K\"ahler structure on $X$
and the holomorphic fibration $f:X\to C$, see for example \cite[Lemma 2.4]{Campana_NIL}. Thus the first map is surjective; which means any real character of $N$ which is invariant under the conjugation action of $Q$ can be extended to a real character of $\pi_1(X)$.

From the above construction, we have
\[
\begin{tikzcd}
X_\ab \ar[rr,"r"] \ar[dd,"\Gal=A"'] &&
C_0 \ar[dd,"\Gal=A"]\\
\\
X \ar[rr,"f"] && C
\end{tikzcd}
\]
This allows us to compare \cref{eq: ses of hol fibration} with the following short exact sequence
\begin{equation}\label{eq: ses from D to A}
1\to \ker(q)\to  D\xrightarrow{q} A\to 1.
\end{equation}
Note that $\rho_\ab: \Gamma\to A$ factors through $Q$, i.e.,
$$\rho_\ab=\rho_Q\circ f_*^\orb,\quad \text{for }\rho_Q:Q\to A.$$
Let $G\coloneqq \rho(N)$, then $G\le D$ is a normal subgroup. Indeed, for any $\rho(n)\in G$ and $d\in D$, since $\rho: \Gamma\to D$ is surjective, we have $d=\rho(\gamma)$ for some $\gamma\in \Gamma$, then
$$d\rho(n)d^{-1}=\rho(\gamma n\gamma^{-1})\in \rho(N).$$
Since $\rho_\ab(N)=0$, we have $G\le \ker(D\to A)$; this means $D/G\twoheadrightarrow A$ which is infinite abelian, hence $[D:G]=\infty$. In conclusion, we obtain
\[
\begin{tikzcd}
1\ar[r] &N\ar[r]\ar[d,"\rho|_N"', two heads] &\Gamma \ar[rr,"f_*^\orb"]\ar[rrdd,dashed,"\rho_\ab", two heads] \ar[dd,"\rho"'] &&
Q=\pi_1^\orb(C) \ar[dd,"\rho_Q",two heads]\ar[r] & 1\\
&G\ar[d,hook]\ar[dr,hook,dashed]\\
1\ar[r] &\ker(q)\ar[r]&D\ar[rr,"\text{free-abelian}"',"q"] && A\ar[r]& 1
\end{tikzcd}
\]
where the two rows are short exact sequences from \cref{eq: ses of hol fibration} and $\rho$.

We divide this situation further into two cases:
\begin{itemize}
    \item The group $G$ is finite; then because $D$ is torsion free, $G$ is trivial.

    Consider the image of $\ker(\rho)$ under $f_*^\orb$, which is also of infinite index in $Q$ since 
    $$D\cong D/G\cong \frac{\Gamma/\ker\rho}{N/(N\cap \ker\rho)}\cong \frac{\Gamma}{N\ker\rho}\cong \frac{\Gamma/N}{N\ker\rho/N}\cong Q/f^\orb_*(\ker \rho).$$
    Taking the corresponding coverings $C_1\to C$ with $\Gal(C_1/C)\cong D$ infinite, we get a proper map $X_\rho\to  C_1$ with connected fibers. Since $C_1$ is an open Riemann surface which is Stein, we know $X_\rho$ is holomorphically convex.
    \item The group $G$ is infinite.

    
    This means that $G$ is an infinite normal subgroup of $D$ (of infinite index).

    We claim that $(G/[G,D])\otimes \Q\ne 0$ with argument similar to that around \cref{eq: quillen}. Let $V=G_\ab\otimes \Q$ be nonzero because $G$ is a torsion-free finitely generated nilpotent group which is infinite. Treat $V$ as a $\Q[D]$-module, and denote $I$ as the augmentation ideal of $\Q[D]$. Since $D$ is nilpotent, we have $I^nV=0$ for some $n$ large, which implies $V/(I\cdot V)\ne 0$. But $V/IV\cong (G/[G,D])\otimes\Q$, and we get the desired conclusion.

    Write $H=(G/[G,D])/\Tor\cong \Z^n$ with $n\ge 1$, and $K=\ker(G \to H)$. Since $K$ is the preimage of the torsion subgroup of $G/[G,D]$, which is $D$-invariant, $K$ is a normal subgroup of $D$. Set $Q_0=D/K$ and consider $\rho_0:\Gamma\xrightarrow{\rho} D\twoheadrightarrow Q_0$, and put $N_0=\ker(\rho_0)$. Denote $X_0=\widetilde X/N_0$ to be the covering corresponding to $\rho_0$.

    Take \(C_1^{\mathrm{orb}}\to C^{\mathrm{orb}}\) to be the orbifold covering with Galois group $D/G\cong Q/f_*^\orb(\ker(\rho))$, and write \(C_1\) for its underlying open Riemann surface. Define \(X_1=\widetilde X/\rho^{-1}(G)\). This smooth cover is the normalization of the reduced fiber product \(X\times_C C_1\). Then $\Gal(X_0/X)\cong D/K$ and $\Gal(X_0/X_1)\cong G/K\cong H$.
    \[
    \begin{tikzcd}
    X_\rho\ar[rr,"\Gal=K"] & &X_0\ar[rr,"\Gal=H=G/K","p_{01}"']\ar[rrd,"\phi"'] && X_1\ar[rr,"\Gal=D/G"]\ar[d,"f_1"] && X\ar[d,"f"]\\
    &&&& C_1\ar[rr,"\Gal=D/G"] && C    
    \end{tikzcd}
    \]
    Define $\phi: X_0\to C_1$ as the composition, then the general fiber of $\phi$ is connected: in fact, for \(t_1\in C_1\) over a regular value \(t\in C\), identify \(f_1^{-1}(t_1)\) with \(F=f^{-1}(t)\). The restriction of the \(H\)-cover \(X_0\to X_1\) to this fiber has monodromy \(\pi_1(F)\twoheadrightarrow N\twoheadrightarrow G/K=H\). 

    Let $\ell_i:H\cong \Z^n\to \Z$ be the $i$-th coordinate of $H$, consider $\ell_i\circ (\rho_0|_N):N\to \Z$; which is well-defined because $\rho(N)=G$ implies that $\rho_0(N)\le G/K=H$. Since $H$ is central in $D/K=Q_0$ by construction, we know that $\ell_i\circ (\rho_0|_N)$ in invariant under conjugate-action of $\Gamma$; by \cref{eq: LHS les}, it extends to a character
    $$\chi_i:\Gamma\to \R.$$
    Let $u_i$ be a real-valued pluriharmonic primitive function of $\chi_i$ on the universal covering $\widetilde X$, satisfying $u_i(\gamma.x)=u_i(x)+\chi_i(\gamma)$ for any $x\in \widetilde X$. These $u_i$'s do not necessarily descend to $X_0$. Consider $\chi_i|_{N_0}$, since $N/N_0\cap N\cong \rho_0(N)=G/K$ which implies $\chi_i|_{N\cap N_0}\equiv 0$. Considering
    $$1\to N\cap N_0\to N_0\xrightarrow{\phi_*^\orb} Q_1\coloneqq \pi_1^\orb(C_1)\to 1,$$
    this means that $\chi_i|_{N_0}$ factors through $Q_1$ to get $\eta_i':Q_1\to \R$. Note that the torsion elements of $Q_1$ are mapped by $\eta_i'$ to 0, thus we can further show $\chi_i$ factors through $\phi_*: N_0\to \pi_1(C_1)$, and write $\eta_i:\pi_1(C_1)\to \R$. Since $C_1$ is Stein, by Cartan's theorem B and standard holomorphic de Rham theory, the cohomology class of $\eta_i$ can be represented by a holomorphic one-form $\omega_i$. Let $v_i$ be the real part of the primitive function on $\widetilde C_1$ of $\omega_i$, and denote $\widetilde \phi: \widetilde X\to \widetilde C_1$ the lift of $\phi$ to the universal coverings. Set
    $$w_i\coloneqq u_i-v_i\circ \widetilde \phi$$
    For any $\gamma\in N_0$ and $x\in \widetilde X$, we have
    $$w_i(\gamma x)-w_i(x)=\chi_i(\gamma)-\phi^*\eta_i(\gamma)=0$$
    Thus $w_i$ descends to $X_0$. Moreover, for any $h\in H=G/K\le D/K$, we have for $x\in X_0$ that
    \begin{equation}\label{eq: Galois action on wi}
    w_i(h x)=w_i(x)+\ell_i(h).
    \end{equation}
     
    Since $C_1$ is Stein, take a nonnegative smooth strictly plurisubharmonic exhaustion function $\psi$ on $C_1$, and define
    $$F=\psi\circ \phi+\sum_{i=1}^n w_i^2: X_0\to \R.$$
    Since each $w_i$ is pluriharmonic, we know $F$ is plurisubharmonic. Next, we show it is exhaustive. For any $r\in \R$, 
    $$F^{-1}([0,r])\subset \phi^{-1}(\{x\in C_1:\psi(x)\le r\})\mathrel{=:} \phi^{-1}(S)$$
    We know $S$ is compact. The $f_1$ is proper by construction. For each $i$, $w_i^2\le r$. Choose a compact set \(T\subset p_{01}^{-1}(f_1^{-1}(S))\) such that \(H\cdot T\) covers the full preimage. Each $w_i$ is bounded on $T$. By \cref{eq: Galois action on wi}, bounds on all $w_i$ allow only finitely many $h\in H$. Therefore, the sublevel set is contained in finitely many translates of $T$. Together, these facts imply the sublevel set $F^{-1}([0,r])$ is compact for each $r\in \R$; i.e., $F$ is an exhaustion function.

    Finally, we study the degeneracy locus of the Levi form of $F$:
    $$\sqrt{-1}\partial\pbar F=\phi^*(\sqrt{-1}\partial\pbar\psi)+2\sum_{i=1}^n\sqrt{-1}\partial w_i\wedge \pbar w_i.$$
    Since $\psi$ is strictly plurisubharmonic, we know $F$ is strictly plurisubharmonic outside of 
    $$Z\coloneqq\{x\in X_0: d\phi\wedge \partial w_i=0,\text{ for all }1\le i\le n\}$$
    which is complex analytic because $\phi$ is holomorphic and each $w_i$ is pluriharmonic. Note that $Z$ must be a proper subset of $X_0$; otherwise, each $w_i$ is constant along the fibers of $\phi$, contradicting \cref{eq: Galois action on wi}; note here we use the fiber of $\phi$ is connected. By \cref{lem: theorem of narasimhan}, we have $X_0$ is holomorphically convex. The existence of a generically strictly plurisubharmonic exhaustion function $F$ shows that the Cartan--Remmert reduction must be of dimension two; now we can apply \cref{prop: CR reduction of dim 2} to get $X_\rho$ is holomorphically convex.
\end{itemize}

\section{On the number of ends of Malcev coverings}\label{sec: On the number of ends of Malcev coverings}
In this section, we show that the Malcev covering of any compact K\"ahler manifold $X$ has at most one end. It is well known that any K\"ahler group can have at most one end (see \cite{Gromov},\cite{ABR}), and this can be seen as a ``nilpotent'' analogue of this result.

We first recall the torsion-free nilpotent completion
of a finitely generated group (see Appendix A of
\cite{Campana_NIL}). Let $\Gamma$ be a finitely
generated group, and consider its descending central
series
\[
\Gamma=\Gamma_1\supseteq\Gamma_2
\supseteq\cdots\supseteq\Gamma_s
\supseteq\cdots,
\]
where $\Gamma_{s+1}=[\Gamma,\Gamma_s]$ for
each $s\geq 1$. Set
\[
\Gamma_s'
\coloneqq
\left\{
g\in\Gamma:
g^n\in\Gamma_s
\text{ for some integer }n\geq 1
\right\},
\]
and let $\Gamma_\infty'
\coloneqq\bigcap_{s\geq 1}\Gamma_s'$. We denote by
 $\Gamma_0^{\nilp}\coloneqq\Gamma/\Gamma_\infty'$ the maximal residually torsion-free nilpotent
quotient of $\Gamma$, and by $\rho_0^{\nilp}:
\Gamma\to\Gamma_0^{\nilp}$ the natural quotient homomorphism.

Let $Y$ be a connected topological space
admitting a universal covering, with
$\Gamma=\pi_1(Y)$ finitely generated. The regular covering $Y^{\nilp}\coloneqq\widetilde Y/\Gamma_\infty'\to Y$ is called the \textit{Malcev covering} of $Y$. The subgroup $K_\Gamma^\Mal\coloneqq \Gamma_\infty'$ is called the \textit{Malcev kernel} of $\Gamma$.

It was proved by Katzarkov in \cite{Katzarkov97} that the Malcev covering of a smooth complex projective manifold is holomorphically convex. This result was subsequently extended to compact K\"ahler manifolds by Leroy \cite{LeroyThesis} (see also \cite{Claudon}). Another proof, based on flat connections, was given by Brudnyi \cite{Brudnyi99,Brudnyi}, whose approach has inspired some of the arguments used in our previous section.

The relevant results of Katzarkov, Leroy, and
Brudnyi can be summarized in the following form.
\begin{proposition}
[Katzarkov, Leroy, Brudnyi]\label{thm: KLB}
Let $X$ be a connected compact K\"ahler manifold, and $\alpha:X\to \Alb(X)$ be its Albanese map. Let $f:X\to S$ be the Stein factorization of $\alpha$, then the Malcev kernel $K_{\pi_1(S)}^\Mal$ of $S$ pulls back under $f$ to $K_{\pi_1(X)}^\Mal$, the Malcev covering $S^\Mal$ of $S$ is Stein, and the induced map
$$f^\Mal: X^\Mal\to S^\Mal$$
is the Cartan--Remmert reduction. In particular, $X^\Mal$ is holomorphically convex.
\end{proposition}

With this property at hand, we will prove the following theorem.
\begin{theorem}\label{thm: one end malcev}
Let $X$ be any compact K\"ahler manifold with fundamental group $\Gamma$, then its Malcev covering has at most one end; more precisely,
\[
e(X^{\Mal})=
\begin{cases}
0, & \text{\quad if }b_1(X)=0,\\
1, & \text{\quad if }b_1(X)>0.
\end{cases}
\]
\end{theorem}
\begin{proof}
Keep the notation given in \cref{thm: KLB}. If $b_1(X)=0$, then the Albanese map of $X$ and $S$ are points. Then we have $X^\Mal=X$ which is compact, and $e(X^\Mal)=0$.

Suppose now that $b_1(X)>0$. Since the proper surjective map $f^\Mal$ with connected fibers preserves the number of ends, it suffices to determine the number of ends of $S^\Mal$. We distinguish two cases.
\begin{itemize}
\item \textup{\textbf{Case 1:}} $\dim S=1$.
Since $S$ is a normal compact complex curve,
it is a smooth compact Riemann surface. The nonconstant holomorphic map $h:S\to\Alb(X)$ implies that $g(S)\ge 1$.

By the residual torsion-free nilpotence of
surface groups \cite{Baumslag}, we have $K_{\pi_1(S)}^\Mal=\{1\}$, i.e., the Malcev covering of $S$ is exactly its universal covering $\widetilde S$. By the uniformization theorem, $\widetilde S$ is biholomorphic to either $\C$ or the unit disk $\mathbb D$. In either case, $S^\Mal$ has exactly one end.
\item \textup{\textbf{Case 2:}} $\dim S\geq2$.

Since $S^\Mal$ is a connected normal Stein space of complex dimension at least two, it has exactly one end by \cite[Corollary 4.10]{BuanicuaStuanuac}.
\end{itemize}
In both cases, $S^\Mal$ has exactly
one end. Since $f^\Mal$ is proper,
surjective, and has connected fibers,
the same holds for $X^\Mal$.
\end{proof}

\begin{example}\label{ex: Hopf}
We know that if $\Gamma$ is a Kähler group, then $e(\Gamma_0^\Mal)\leq1$. The K\"ahler assumption is essential. Indeed, let $X$ be a  Hopf surface. Then $X$ is a compact non-K\"ahler complex surface with $\pi_1(X)\cong\Z$. Consequently, its Malcev covering coincides
with its universal covering
$$X^\Mal\cong\C^2\setminus\{0\},$$
which has exactly two ends.
\end{example}

By one direction of Stallings' theorem on ends \cite[Theorem 9.22(3)]{DrutuKapovich}, a finitely generated group which admits a nontrivial splitting over a finite subgroup has more than one end. Consequently:
\begin{corollary}\label{cor: malcev kahler obstruction}
If $\Gamma$ is a K\"ahler group, and $\Gamma^\Mal$ is nontrivial, then $\Gamma^\Mal$ does not admit any nontrivial free-product decomposition.
\end{corollary}

\begin{remark}
The free indecomposability of $\Gamma^\Mal$ in \cref{cor: malcev kahler obstruction} can also be deduced from the known restrictions on the first resonance varieties of K\"ahler groups, together with the fact that the degree-one cup product is determined by the rational two-step Malcev quotient; see \cite[Theorem 4.1]{DimcaSuciu09} and \cite[\S 2]{Amoros96}.
\end{remark}

\section*{Acknowledgments}
I would like to thank Mohan Ramachandran for communicating the ideas used in the prove \cref{prop: abelian cover} to me. I am also grateful to Ya Deng for helpful discussions and for his interest in this work.

The author was partially supported by the China Postdoctoral Science Foundation (Grant No.~2023M743396).

\section*{Declaration of AI Assistance}
During the revision of this manuscript, an AI tool was used to assist in identifying and addressing a gap in the original argument, and the corresponding part of the proof was rewritten accordingly. AI assistance was also used for proofreading and language editing to improve the clarity of the exposition. All other mathematical arguments and the overall structure of the original manuscript remain unchanged. The author has independently reviewed and verified the revised argument and takes full responsibility for the mathematical correctness and content of the manuscript.

\bibliographystyle{ssmfalpha} 
\bibliography{main}

\end{document}